\documentclass[12pt]{article}
\usepackage[UKenglish]{babel}
\usepackage[T1]{fontenc}
\usepackage{lmodern,amsmath,amsthm,amsfonts,amssymb,graphicx,float,microtype,thmtools,underscore,mathtools,thm-restate}
\usepackage[shortlabels]{enumitem}
\setlist[itemize]{topsep=0ex,itemsep=0ex,parsep=0ex}
\setlist[enumerate]{topsep=0ex,itemsep=0ex,parsep=0ex}
\usepackage[usenames,dvipsnames,svgnames,table]{xcolor}
\usepackage[unicode=true]{hyperref}
\hypersetup{ 
colorlinks,
linkcolor={blue!60!black},
citecolor={black},
urlcolor={blue!60!black},
pdftitle={The Excluded Tree Minor Theorem Revisited}}
\usepackage[capitalise, compress, nameinlink, noabbrev]{cleveref}
\crefname{lem}{Lemma}{Lemmas}
\crefname{thm}{Theorem}{Theorems}
\crefname{cor}{Corollary}{Corollaries}
\crefname{enumi}{Item}{Items}
\crefformat{equation}{(#2#1#3)}
\Crefformat{equation}{Equation #2(#1)#3}
\crefformat{enumi}{#2#1#3}
\Crefformat{enumi}{Item (#2#1#3)}
\newcommand{\defn}[1]{\textcolor{Maroon}{\emph{#1}}}
\usepackage[longnamesfirst,numbers,sort&compress]{natbib}
\makeatletter
\def\NAT@spacechar{~}
\makeatother
\usepackage[margin=28mm]{geometry}
\renewcommand{\baselinestretch}{1.1}
\allowdisplaybreaks

\DeclarePairedDelimiter{\set}{\{}{\}} 
\newcommand\subsetcong{\mathrel{\text{%
    \setbox0\hbox{$\subseteq$}%
    \rlap{\hbox to \wd0{\hss\hss\hss\raisebox{1.5\height}{$\sim$}\hss}}\box0
}}}

\renewcommand{\epsilon}{\varepsilon}
\renewcommand{\emptyset}{\varnothing}

\renewcommand{\geq}{\geqslant}
\renewcommand{\leq}{\leqslant}

\DeclareMathOperator{\dist}{dist}
\DeclareMathOperator{\tw}{tw}

\DeclareMathOperator{\pw}{pw}
\DeclareMathOperator{\td}{td}

\newcommand{\PP}{\mathcal{P}}
\newcommand{\DD}{\mathcal{D}}
\newcommand{\BB}{\mathcal{B}}
\newcommand{\FF}{\mathcal{F}}

\newcommand{\NN}{\mathbb{N}}
\newcommand{\OO}{\mathcal{O}}

\newcommand{\david}[1]{{\color{orange} DW: #1}}

\newcommand{\piotr}[1]{\textcolor{brown}{PM: #1}}
\newcommand{\pat}[1]{\textcolor{DarkGrey}{PaMo: #1}}
\newcommand{\gwen}[1]{\textcolor{cyan}{GJ: #1}}

\theoremstyle{plain}
\newtheorem{thm}{Theorem}
\newtheorem{lem}[thm]{Lemma}

\newtheorem{prop}[thm]{Proposition}
\newtheorem{obs}[thm]{Observation}
\newtheorem*{claim}{Claim}
\crefname{obs}{Observation}{Observations}
\newtheorem*{lem*}{Lemma}
\theoremstyle{definition}

\newtheorem*{conj*}{Conjecture}

\begin{document}
\title{\bf\boldmath
\fontsize{18pt}{18pt}\selectfont
The Excluded Tree Minor Theorem Revisited}

\author{
\fontsize{16pt}{16pt}\selectfont
Vida Dujmovi{\'c}\,\footnotemark[6]\qquad
Robert~Hickingbotham\,\footnotemark[2] \qquad
Gwena\"el Joret\,\footnotemark[4] \\
Piotr Micek\,\footnotemark[5] \qquad
Pat Morin\,\footnotemark[3] \qquad
David~R.~Wood\,\footnotemark[2]
}

\maketitle

\begin{abstract}


We prove that for every tree $T$ of radius $h$, there is an integer $c$ such that every $T$-minor-free graph is contained in $H\boxtimes K_c$ for some graph $H$ with pathwidth at most $2h-1$. This is a qualitative strengthening of the Excluded Tree Minor Theorem of Robertson and Seymour (GM I). We show that radius is the right parameter to consider in this setting, and $2h-1$ is the best possible bound. 
\end{abstract}

\footnotetext[2]{School of Mathematics, Monash University, Melbourne, Australia (\texttt{\{robert.hickingbotham, david.wood\}@monash.edu}). Research of Wood supported by the Australian Research Council. Research of Hickingbotham supported by Australian Government Research Training Program Scholarships.}

\footnotetext[6]{School of Computer Science and Electrical Engineering, University of Ottawa, Ottawa, Canada (\texttt{vida.dujmovic@uottawa.ca}). Research supported by NSERC, and a Gordon Preston Fellowship from the School of Mathematics at Monash University.}

\footnotetext[4]{D\'epartement d'Informatique, Universit\'e libre de Bruxelles, Belgium (\texttt{gwenael.joret@ulb.be}). Supported by the Australian Research Council, and by a CDR grant and a PDR from the National Fund for Scientific Research (FNRS).}

\footnotetext[3]{School of Computer Science, Carleton University, Ottawa, Canada (\texttt{morin@scs.carleton.ca}). Research  supported by NSERC.}

\footnotetext[5]{Faculty of Mathematics and Computer Science, Jagiellonian University, Kraków, Poland (\texttt{piotr.micek@uj.edu.pl}). }

\section{\Large Introduction}
\label{Introduction}


\citet{RS-I} proved that for every tree $T$ there is an integer $c$ such that every $T$-minor-free graph has pathwidth at most $c$. \citet{BRST91} and \citet{Diestel95} showed the same result with $c=|V(T)|-2$, which is best possible, since the complete graph on $|V(T)|-1$ vertices is $T$-minor-free and has pathwidth $|V(T)|-2$. Graph product structure theory describes graphs in complicated classes as subgraphs of products of simpler graphs \citep{DJMMUW20,ISW,UTW}. Inspired by this viewpoint, we prove the following result, where $H\boxtimes K_c$ is the graph obtained from $H$ by replacing each vertex of $H$ by a copy of $K_c$ and replacing each edge of $H$ by the join between the corresponding copies of $K_c$.


\begin{thm}
\label{ExcludedTree}
For every tree $T$ of radius $h$, there exists  $c\in\NN$ such that every $T$-minor-free graph $G$ is contained in $H\boxtimes K_c$ for some graph $H$ with pathwidth at most $2h-1$. 
\end{thm}

\cref{ExcludedTree} is a qualitative strengthening of the above-mentioned result of \citet{RS-I} since $\pw(G)\leq \pw(H\boxtimes K_c) \leq c(\pw(H)+1)-1 \leq 2ch-1$. Note that the proof of \cref{ExcludedTree} depends on the above-mentioned result of \citet{RS-I}. The point of \cref{ExcludedTree} is that $\pw(H)$ only depends on the radius of $T$, not on $|V(T)|$ which may be much greater than the radius. Moreover, radius is the right parameter of $T$ to consider here, as we now show. 


For a tree $T$, let $g(T)$ be the minimum $k\in\NN$ such that for some $c\in\NN$ every $T$-minor-free graph $G$ is contained in $H\boxtimes K_c$ where $\pw(H)\leq k$. \cref{ExcludedTree} shows that if $T$ has radius $h$, then $g(T)\leq 2h-1$. 
Now we show a lower bound. The following lemma by \citet{UTW} is useful, where $T_{h,d}$ is the complete $d$-ary tree of radius $h$.

\begin{lem}[{\protect\citep[v1,~Proposition~56]{UTW}}]
\label{CompleteTreePathwidthH}
For any $h,c\in\NN$, there exists $d\in\NN$ such that for every graph $H$, if $T_{h,d}$ is contained in $H\boxtimes K_c$, then $\pw(H)\geq h$. 
\end{lem}

Let $T$ be any tree with radius $h$. Thus $T$ contains a path on $2h$ vertices, and $T_{h-1,d}$ contains no $T$-minor, as otherwise $T_{h-1,d}$ would contain a path on $2h$ vertices. 
By \cref{CompleteTreePathwidthH}, if $T_{h-1,d}$ is contained in $H\boxtimes K_c$, then $\pw(H)\geq h-1$. Hence 
\begin{equation}
\label{Summary}
h-1 \leq g(T) \leq 2h-1.
\end{equation}
This says that the radius of $T$ is the right parameter to consider in \cref{ExcludedTree}. 

Moreover, both the lower and upper bounds in \cref{Summary} can be achieved, as we now explain. The upper bound in \cref{Summary} is achieved when $T$ is a complete ternary tree, as shown by the  following result.

\begin{restatable}{prop}{TreeLowerBound}
\label{TreeLowerBound}
For all $h,c\in \NN$, there is a $T_{h,3}$-minor-free graph $G$, such that for every graph $H$, if $G$ is contained in $H\boxtimes K_c$, then $H$ has a clique of size $2h$, implying $\pw(H)\geq\tw(H)\geq 2h-1$.
\end{restatable}

The next result improves \cref{ExcludedTree} for an excluded path. It shows that the lower bound in \cref{Summary} is achieved when $T$ is a path, since $P_{2h+1}$ has radius $h$, and a graph has no path on $2h+1$ vertices if and only if it is $P_{2h+1}$-minor-free. 

\begin{prop}
\label{ExcludedPath}
For any $h\in\NN$, every graph $G$ with no path on $2h+1$ vertices is contained in $H\boxtimes K_{4h}$ for some graph $H$ with $\pw(H) \leq h-1$. 
\end{prop}



\section{\Large Background}
\label{Background}

We consider simple, finite, undirected graphs~$G$ with vertex-set~${V(G)}$ and edge-set~${E(G)}$. See \citep{Diestel5} for graph-theoretic definitions not given here. 
For $m,n \in \mathbb{Z}$ with $m \leq n$, let $[m,n]:=\{m,m+1,\dots,n\}$ and $[n]:=[1,n]$. 

A graph $H$ is a \defn{minor} of a graph $G$ if $H$ is isomorphic to a graph that can be obtained from a subgraph of $G$ by contracting edges. A graph~$G$ is \defn{$H$-minor-free} if~$H$ is not a minor of~$G$. 
An \defn{$H$-model} in a graph $G$ consists of pairwise-disjoint vertex subsets  $(W_x \subseteq V(G) :x\in V(H))$ (called \defn{branch sets})  such that each subset induces a connected subgraph of $G$, and 
for each edge $xy\in V(H)$ there is an edge in $G$ joining $W_x$ and $W_y$. Clearly $H$ is a minor of $G$ if and only if $G$ contains an $H$-model.

A \defn{tree-decomposition} of a graph $G$ is a collection $(B_x :x\in V(T))$ of subsets of $V(G)$ (called \defn{bags}) indexed by the vertices of a tree $T$, such that (a) for every edge $uv\in E(G)$, some bag $B_x$ contains both $u$ and $v$, and (b) for every vertex $v\in V(G)$, the set $\{x\in V(T):v\in B_x\}$ induces a non-empty (connected) subtree of $T$. 
The \defn{width} of $(B_x:x\in V(T))$ is $\max\{|B_x| \colon x\in V(T)\}-1$. The \defn{treewidth} of a graph $G$, denoted by \defn{$\tw(G)$}, is the minimum width of a tree-decomposition of $G$. A \defn{path-decomposition} is a tree-decomposition in which the underlying tree is a path, simply denoted by the sequence of bags $(B_1,\dots,B_n)$. The \defn{pathwidth} of a graph $G$, denoted by \defn{$\pw(G)$}, is the minimum width of a path-decomposition of $G$. 



The following lemma is folklore (see \citep{ISW} for a proof). 

\begin{lem}
\label{HittingSet}
For every graph $G$, for every tree-decomposition $\DD$ of $G$, for every collection $\FF$ of connected subgraphs of $G$, and for every $\ell\in\NN$, either\textnormal{:}
\begin{enumerate}[\textnormal{(}a\textnormal{)}]
    \item there are $\ell$ vertex-disjoint subgraphs in $\FF$, or
    \item there is a set $S\subseteq V(G)$ consisting of at most $\ell-1$ bags of $\DD$ such that $S \cap V(F) \neq \emptyset$ for all $F \in \FF$.
\end{enumerate}
\end{lem}

The \defn{strong product} of graphs~$A$ and~$B$, denoted by~${A \boxtimes B}$, is the graph with vertex-set~${V(A) \times V(B)}$, where distinct vertices ${(v,x),(w,y) \in V(A) \times V(B)}$ are adjacent if
	${v=w}$ and ${xy \in E(B)}$, or
	${x=y}$ and ${vw \in E(A)}$, or
	${vw \in E(A)}$ and~${xy \in E(B)}$.

Let $G$ be a graph. A \defn{partition} of $G$ is a set $\PP$ of sets of vertices in $G$ such that each vertex of $G$ is in exactly one element of $\PP$. Each element of $\PP$ is called a \defn{part}. The \defn{width} of $\PP$ is the maximum number of vertices in a part. The \defn{quotient} of $\PP$ (with respect to $G$) is the graph, denoted by \defn{$G/\PP$}, with vertex set $\PP$ where distinct parts $A,B\in \PP$ are adjacent in $G/\PP$ if and only if some vertex in $A$ is adjacent in $G$ to some vertex in $B$. An \defn{$H$-partition} of $G$ is a partition $\PP$ of $G$ such that $G/\PP$ is contained in $H$. The following observation connects partitions and products.

\begin{obs}[\citep{DJMMUW20}]
\label{ObsPartitionProduct}
For all graphs $G$ and $H$ and any $p\in\NN$, $G$ is contained in $H\boxtimes K_p$ if and only if $G$ has an $H$-partition with width at most $p$.
\end{obs}

\section{\Large Proofs}
\label{Proofs}

We prove the following quantitative version of \cref{ExcludedTree}.

\begin{thm}
\label{ExcludedTreePrecise}
Let $T$ be a tree with $t$ vertices, radius $h$, and maximum degree $d$. Then every $T$-minor-free graph $G$ is contained in $H\boxtimes K_{(d+h-2)(t-1)}$ for some graph $H$ with pathwidth at most $2h-1$.
\end{thm}

Recall that $T_{h,d}$ is the complete $d$-ary tree of radius $h$. \cref{ObsPartitionProduct} and the next lemma imply \cref{ExcludedTreePrecise}, since  the tree $T$ in \cref{ExcludedTreePrecise} is a subtree of $T_{h,d}$, and every $T$-minor-free graph $G$ satisfies $\tw(G)\leq\pw(G)\leq t-2$ by the result of \citet{BRST91} mentioned in \cref{Introduction}.

\begin{lem}
For any $h,d\in\NN$ with $d+h\geq 3$, for every $T_{h,d}$-minor-free graph $G$, for every tree-decomposition $\DD$ of $G$, and for every vertex $r$ of $G$, the graph $G$ has a partition $\PP$ such that: 
\begin{itemize}
\item each part of $\PP$ is a subset of the union of at most $d+h-2$ bags of $\DD$, 
\item $\{r\}\in\PP$, and
\item $G/\PP$ has a path-decomposition of width at most $2h-1$ in which the first bag contains $\{r\}$. 
\end{itemize}
\end{lem}

\begin{proof} 
We proceed by induction on pairs $(h,|V(G)|)$ in a lexicographic order. 
Fix $h$, $d$, $G$, $\DD$, and $r$ as in the statement. 
We may assume that $G$ is connected. The statement is trivial if $|V(G)|\leq 1$. Now assume that $|V(G)|\geq 2$. 

For the base case, suppose that $h=1$. For $i\geq 0$, let $V_i:=\{v\in V(G):\dist_G(v,r)=i\}$. So $V_0=\{r\}$. If $|V_i|\geq d$ for some $i\geq 1$, then contracting $G[V_0\cup\dots\cup V_{i-1}]$ into a single vertex gives a $T_{1,d}$ minor. So $|V_i|\leq d-1=d+h-2$ for each $i\geq 0$. 
Thus $\PP:=(V_i:i\geq 0)$ is a partition of $G$, and each part of $\PP$ is a subset of the union of at most $d+h-2$ bags of $\DD$. Moreover, the quotient $G/\PP$ is a path, which has a path-decomposition of width 1, in which the first bag  contains $\{r\}$. 

Now assume that $h\geq 2$ and the result holds for $h-1$. 
Let $R$ be the neighbourhood of $r$ in $G$.
Let $\mathcal{F}$ be the set of all connected subgraphs of $G-r$ that contain a vertex from $R$ and contain a $T_{h-1,d+1}$ minor. If there are $d$ pairwise vertex-disjoint subgraphs $S_1,\ldots,S_d$ in $\mathcal{F}$, 
then we claim that $G$ contains a $T_{h,d}$ minor. 
Indeed, for each $i\in [d]$ consider a $T_{h-1,d+1}$-model $(W^i_x: x\in V(T_{h-1,d+1}))$ in $S_i$. Since $S_i$ is connected, we may assume that all vertices of $S_i$ are in the model. 
For each $i\in [d]$, let $y_i$ be a node of $T_{h-1,d+1}$ such that $W^i_{y_i}$ contains a vertex from $R$, and let $Y^i$ be the union of $W^i_x$ for all ancestors $x$ of $y_i$ in $T_{h-1,d+1}$. 
Observe that there is a $T_{h-1,d}$-model in $S_i$ such that the root of $T_{h-1,d}$ is mapped to the set $Y^i$. Therefore $G-r$ contains $d$ pairwise disjoint models of $T_{h-1,d}$ such that each root branch set contains a vertex from $R$. So $G$ contains a model of $T_{h,d}$, as claimed.

So $\mathcal{F}$ contains no $d$ pairwise vertex-disjoint elements. By \cref{HittingSet}, there is a minimal set $X\subseteq V(G-r)$, such that $X$ is a subset of the union of $d-1\leq d+h-2$ bags of $\DD$, and $G-r-X$ contains no element of $\mathcal{F}$. 

Let $G_1,\dots,G_p$ be the components of $G-r-X$ that contain a vertex from $R$. 
By construction of $X$, the graph $G_i$ contains no $T_{h-1,d+1}$ minor. By induction, 
$G_i$ has a partition $\PP_i$ such that: 
\begin{itemize}
\item each part of $\PP_i$ is a subset of the union of at most $(d+1)+(h-1)-2=d+h-2$ bags of $\DD$, and 
\item $G_i/\PP_i$ has a path-decomposition $\mathcal{B}_i$ of width at most $2h-3$.
\end{itemize}

Let $Z:= V(G-r-X)\setminus V(G_1\cup\dots\cup G_p)$; that is, $Z$ is the set of vertices of all components of $G-r-X$ that have no vertex in $R$.

Consider a vertex $v\in X$. By the minimality of $X$, the graph $G-r-(X\setminus\{v\})$ contains a connected subgraph $Y_v$ that contains $v$ and a vertex $r_v\in R$ (and contains a $T_{h-1,d+1}$ minor). Let $P_v$ be a path from $v$ to $r_v$ in $Y_v$ plus the edge $r_vr$. So $P_v-\{v,r\}$ is contained in some $G_i$, and thus $P_v$ avoids $Z$. So 
$\cup\{P_v:v\in X\}$ is a connected subgraph in $G-Z$. Let $G'$ be obtained from $G$ by contracting $\cup\{P_v:v\in X\}$ into a vertex $r'$, and deleting any remaining vertices not in $Z$. So $V(G')=\{r'\}\cup Z$. Since $G'$ is a minor of $G$, the graph $G'$ is  $T_{h,d}$-minor-free. Let $\DD'$ be the tree-decomposition of $G'$ obtained from $\DD$ by replacing each instance of each vertex in $\cup\{P_v:v\in X\}$ by $r'$ then removing the other vertices in $V(G)\setminus V(G')$.  
Observe that for every bag $B$ in $\mathcal{D'}$, we have $B-\{r'\}$ contained in some bag of $\DD$. 
By induction, $G'$ has a partition $\PP'$ such that: 
\begin{itemize}
\item each part of $\PP'$ is a subset of the union of at most $d+h-2$ bags of $\DD'$, 
\item $\{r'\}\in\PP'$, and
\item $G'/\PP'$ has a path-decomposition $\BB'$ of width at most $2h-1$ in which the first bag contains $\{r'\}$. 
\end{itemize}

Let $\PP:=\{\{r\}\}\cup\{X\}\cup\PP_1\cup\dots\cup\PP_p \cup (\PP'\setminus\{\{r'\}\})$. Then $\PP$ is a partition of $G$ such that each part is a subset of the union of at most $d+h-2$ bags of $\DD$. 
Let $\mathcal{B}$ be a sequence of subsets of vertices of $G/\PP$ obtained from the concatenation of $\mathcal{B}_1,\dots,\mathcal{B}_p,$ and $\BB'$ by adding $\{r\}$ and $X$ to every bag that comes from $\mathcal{B}_1,\dots,\mathcal{B}_p$ and replacing $\{r'\}$ by $X$. 
Now we argue that $\mathcal{B}$ is a path-decomposition of $G/\PP$. 
Indeed, each part of $\PP$ is contained in consecutive bags of $\BB$, specifically $\{r\}$ and $X$ are added to all bags across $\BB_1,\ldots,\BB_p$, and $X$ is in the first bag of $\BB'$. 
Since $G_1,\ldots,G_p$ are components of $G-r-X$,
the neighbourhood in $G/\PP$ of a part in $\PP_i$ is contained in $\PP_i\cup\{\{r\}, X\}$.
Note also that the neigbourhood of $\{r\}$ in $G/\PP$ is contained in $\PP_1\cup\cdots \cup \PP_p\cup \{X\}$. 
It follows that $\BB$ is a path-decomposition of $G/\PP$. By construction, the width of $\BB$ is at most $2h-1$ and the first bag contains $\{r\}$, as required. 
\end{proof}

We now turn to the proof of \cref{ExcludedPath}. We in fact prove a stronger result in terms of tree-depth. A forest is \defn{rooted} if each component has a root vertex (which defines the ancestor relation). The \defn{vertex-height} of a rooted forest $F$ is the maximum number of vertices in a root--leaf path in $F$. The \defn{closure} of a rooted forest $F$ is the graph $G$ with $V(G):=V(F)$  with $vw\in E(G)$ if and only if $v$ is an ancestor of $w$ (or vice versa). The \defn{tree-depth} of a graph $G$ is the minimum vertex-height of a rooted forest $F$ such that $G$ is a subgraph of the closure of $F$. It is well-known and easily seen that $\pw(G)\leq\td(G)-1$ for every graph $G$. Thus, the following lemma implies \cref{ExcludedPath} since every $P_{2h+1}$-minor-free graph $G$ has $\tw(G)\leq\pw(G)\leq 2h-1$ by the result of \citet{BRST91} mentioned in \cref{Introduction}.

\begin{lem}
For any $h,k\in\NN$, for every graph $G$ with no path on $2h+1$ vertices, for every tree-decomposition $\DD$ of $G$, the graph $G$ has a partition $\PP$ such that $\td(G/\PP) \leq h$ and each part of $\PP$ is a subset of at most two bags of $\DD$. 
\end{lem}

\begin{proof}
We proceed by induction on $h$. For $h=1$, $G$ is the disjoint union of copies of $K_1$ and $K_2$. Let $\PP$ be the partition of $G$ where the vertex-set of each component of $G$ is a part of $\PP$. Thus $E(G/\PP)=\emptyset$ and $\td(G/\PP)=1$. Each part is a subset of one bag of $\DD$. 

Now assume $h \geq 2$ and the claim holds for $h-1$. We may assume that $G$ is connected. 
Suppose $G$ contains three vertex-disjoint paths, $P^{(1)}$, $P^{(2)}$ and $P^{(3)}$, each with $2h-1$ vertices. Let $G'$ be the graph obtained by contracting each path  $P^{(i)}$ into a vertex $v_i$. Since $G'$ is connected, there is a $(v_i,v_j)$-path of length at least $2$ in $G'$ for some distinct $i,j\in \{1,2,3\}$. Without loss of generality, $i=1$ and $j=2$. So there exist vertices $u\in V(P^{(1)})$ and $v\in V(P^{(2)})$ together with a $(u,v)$-path $Q$ of length at least $2$ in $G$ that internally avoids $P^{(1)} \cup P^{(2)}$. Let $x$ be the endpoint of $P^{(1)}$ that is furthest from $u$ (on $P^{(1)}$) and let $y$ be the endpoint of $P^{(2)}$ that is furthest from $v$ (on $P^{(2)}$). Then $(xP^{(1)}uQvP^{(2)}y)$ is a path with at least $2h+1$ vertices, a contradiction.

Now assume that $G$ contains no three vertex-disjoint paths with $2h-1$ vertices. By \cref{HittingSet}, there is a set $S\subseteq V(G)$ consisting of at most two bags of $\DD$ such that $G-S$ is $P_{2h-1}$-free. By induction, $G-S$ has a partition $\PP'$ such that $\td((G-S)/\PP') \leq h-1$ and each part of $\PP'$ is a subset of at most two bags of $\DD$. Let $\PP:=\PP'\cup \{S\}$. Then $\PP$ is the desired partition of $G$ since $\td(G/\PP)\leq 
\td( (G-S)/\PP') +1 \leq h$.
\end{proof}

We turn to the proof of Proposition~\ref{TreeLowerBound}. 
It is a strengthening of a similar result by \citet[Lemma~13]{NSSW19}.

\TreeLowerBound*

\begin{proof}
We proceed by induction on $h\geq 1$. First consider the base case $h=1$. Let $G$ be a path on $n= c+1$ vertices. Thus $G$ is $T_{1,3}$-minor-free. Suppose that $G$ is contained in $H\boxtimes K_c$. Since $n>c$ and $G$ is connected, $|E(H)|\geq 1$ and $H$ has a clique of size 2, as desired. 

Now assume $h\geq 2$ and the result holds for $h-1$. Let $t_0:=|V(T_{h-1,3})|$. By induction,  there is a $T_{h-1,3}$-minor-free graph $G_0$, such that for every graph $H$, if $G_0$ is contained in $H\boxtimes K_{c}$, then $H$ has a clique of size $2h-2$. Let $G$ be obtained from a path $P$ of length $c+1$ as follows: for each edge $vw$ of $P$, add $2c$ copies of $G_0$ complete to $\{v,w\}$. 

Suppose for the sake of contradiction that $G$ contains a $T_{h,3}$-model. Let $X$ be the branch set corresponding to the root of $T_{h,3}$. So $G-X$ contains 
three pairwise disjoint subgraphs $Y_1,Y_2,Y_3$, each containing a $T_{h-1,3}$-minor. Each $Y_i$ intersects $P$, otherwise $Y_i$ is contained in some component of $G-P$ which is a copy of $G_0$. By the construction of $G$, each $Y_i$ intersects $P$ in a subpath $P_i$. Without loss of generality, $P_1,P_2,P_3$ appear in this order in $P$. Since each component of $G-P$ is only adjacent to an edge of $P$, no component of $G-P_2$ is adjacent to both $Y_1$ and $Y_3$. In particular, $X$ is not adjacent to both $Y_1$ and $Y_3$, which is a contradiction. Thus $G$ is $T_{h,3}$-minor-free.

Now suppose that $G$ is contained in $H\boxtimes K_c$. Let $\PP$ be the corresponding $H$-partition of $G$. Since $|V(P)|>c$ there is an edge $v_1v_2$ of $P$ with $v_i\in Q_i$ for some distinct parts $Q_1,Q_2\in\PP$. At most $c-1$ of the copies of $G_0$ attached to $v_1v_2$ intersect $Q_1$, and at most $c-1$ of the copies of $G_0$ attached to $v_1v_2$ intersect $Q_2$. Thus some copy of $G_0$ attached to $v_1v_2$ avoids $Q_1\cup Q_2$. Let $H_0$ be the subgraph of $H$ induced by those parts that intersect this copy of $G_0$. So neither $Q_1$ nor $Q_2$ is in $H_0$. By induction, $H_0$ has a clique $C_0$ of size $2(h-1)$. Since $G_0$ is complete to $v_1v_2$, we have that $C_0\cup\{Q_1,Q_2\}$ is a clique of size $2h$ in $H$, as desired. 
\end{proof}

{\fontsize{10pt}{11pt}\selectfont
\bibliographystyle{DavidNatbibStyle}
\bibliography{DavidBibliography}}

\end{document}